\documentclass[12pt]{amsart}
\usepackage{tikz}
\usepackage{amssymb}
\def\rk{\mathop{\rm rk}\nolimits}
\def\vertexradius{.1}
\def\vertex(#1){\fill (#1) circle (\vertexradius)}
\def\myscale{.5}
\def\mywidth{.8pt}

\def\stilde{\tilde{s}}
\def\htilde{\tilde{h}}

\def\Hnbar{\overline{H^n}}
\def\Fbar{\bar{F}} 
\def\L{\Lambda} 

\def\F{\mathbb{F}}     
\def\Q{\mathbb{Q}}     
\def\Z{\mathbb{Z}}     
\def\N{\mathbb{N}}     
\def\R{\mathbb{R}}     
\def\St{\mathfrak{St}} 
\def\PSt{\mathfrak{PSt}} 

\def\g{\mathfrak{g}}   
\def\G{\mathfrak{G}}   
\def\hbar{\bar{h}}     
\def\U{\mathfrak{U}}   

\def\stilde{\tilde{s}}  

\def\tensor{\otimes}
\def\iso{\cong}
\def\freeproduct{\mathop{*}}

\def\sset{\subseteq}

\def\set#1#2{\{#1:#2\}}

\def\mathrlap#1{\mathchoice
{\rlap{$\displaystyle #1$}}%
{\rlap{$\textstyle #1$}}%
{\rlap{$\scriptstyle #1$}}%
{\rlap{$\scriptscriptstyle #1$}}}

\newtheorem{theorem}{Theorem}
\newtheorem{lemma}[theorem]{Lemma}

\newtheorem{corollary}[theorem]{Corollary}
\theoremstyle{remark}
\newtheorem*{remark}{Remark}

\begin{document}

\title{Presentation of Hyperbolic
  Kac--Moody groups over rings}
\author{Daniel Allcock}
\address{Department of Mathematics\\University of Texas, Austin}
\email{allcock@math.utexas.edu}
\urladdr{http://www.math.utexas.edu/\textasciitilde allcock}
\author{Lisa Carbone}
\address{Department of Mathematics\\
Rutgers, The State University of NJ}
\email{carbonel@math.rutgers.edu}
\urladdr{http://www.math.rutgers.edu/\textasciitilde carbonel}
\date{September 20, 2014}

\begin{abstract}
Tits has defined Kac--Moody  and Steinberg groups over commutative
rings, providing
infinite dimensional analogues of the Chevalley--Demazure group schemes.
Here we establish simple explicit presentations for all Steinberg and
Kac--Moody groups whose Dynkin diagrams are hyperbolic and simply
laced.  Our presentations are analogues of the Curtis--Tits
presentation of the finite groups of Lie type.  When the ground ring
is finitely generated, we derive the finite presentability of the
Steinberg group, and similarly for the Kac--Moody group when the
ground ring is a Dedekind domain of arithmetic type.  These
finite-presentation results need slightly stronger hypotheses when the
rank is smallest possible, namely~$4$.  The presentations simplify
considerably when the ground ring is~$\Z$, a case of special interest
because of the conjectured role of the Kac--Moody group $E_{10}(\Z)$
in superstring theory.
\end{abstract}

\maketitle

\section{Introduction}
\label{sec-introduction}

\noindent
Kac--Moody groups are infinite-dimensional generalizations of reductive
Lie groups and algebraic groups.  Over general rings, their final definition
has not yet been found---it should be some sort of generalization of
the Chevalley--Demazure group schemes.  Given any root system, Tits
defined a functor from commutative rings to groups and proved that it approximates any acceptable definition,
and gives the unique best definition when the coefficient ring is a
field \cite{Tits}.  His definition, by generators and relations, is
very complicated.  Even enumerating his relations in
non-affine examples is difficult and in some cases impracticable
(\cite{Carbone-Murray}, \cite{Allcock-prenilpotent}).

In this paper we study the question of improving this in the case of
the simplest non-affine Dynkin diagrams: the simply laced hyperbolic
ones.  The main result is that these, and related groups, have quite
simple presentations, often finite.  Our results parallel those
established for the affine case in \cite{Allcock-affine}.  Near the
end of the introduction we will remark on the situation beyond the
affine and simply-laced hyperbolic cases.

An irreducible Dynkin diagram is called {\it hyperbolic\/} if it is
not of affine or finite dimensional type, but its proper irreducible
subdiagrams are.  It is called {\it simply laced\/} if each pair of
nodes is either unjoined, or joined by a single bond.  One can
classify the simply laced hyperbolic diagrams \cite[\S6.9]{Humphreys},
namely those in table~\ref{tab-diagrams}.  The most important one is
the last, known as $E_{10}$, because of its conjectural role in
superstring theory (see below).  
We will pass
between Dynkin diagrams and their generalized Cartan matrices whenever
convenient.

\begin{table}
\setlength\unitlength{1mm}
\begin{picture}(110,99)(0,-4)
\put(0,2){\makebox[0pt][l]{rank 10}}
\put(15,0){\makebox[0pt][l]{
\begin{tikzpicture}[line width=\mywidth, scale=\myscale]
\draw (0,0) -- (7,0);
\draw (1,0) -- (1,1);
\draw (5,0) -- (5,1);
\vertex (0,0);
\vertex (1,0);
\vertex (2,0);
\vertex (3,0);
\vertex (4,0);
\vertex (5,0);
\vertex (6,0);
\vertex (7,0);
\vertex (1,1);
\vertex (5,1);
\end{tikzpicture}%
}}%
\put(55,0){\makebox[0pt][l]{
\begin{tikzpicture}[line width=\mywidth, scale=\myscale]
\draw (0,0) -- (8,0);
\draw (2,0) -- (2,1);
\vertex (0,0);
\vertex (1,0);
\vertex (2,0);
\vertex (3,0);
\vertex (4,0);
\vertex (5,0);
\vertex (6,0);
\vertex (7,0);
\vertex (8,0);
\vertex (2,1);
\end{tikzpicture}%
}}%
\put(0,15){\makebox[0pt][l]{rank 9}}
\put(15,13){\makebox[0pt][l]{
\begin{tikzpicture}[line width=\mywidth, scale=\myscale]
\draw (0,0) -- (6,0);
\draw (1,0) -- (1,1);
\draw (4,0) -- (4,1);
\vertex (0,0);
\vertex (1,0);
\vertex (2,0);
\vertex (3,0);
\vertex (4,0);
\vertex (5,0);
\vertex (6,0);
\vertex (1,1);
\vertex (4,1);
\end{tikzpicture}%
}}%
\put(50,13){\makebox[0pt][l]{
\begin{tikzpicture}[line width=\mywidth, scale=\myscale]
\draw (0,0) -- (7,0);
\draw (3,0) -- (3,1);
\vertex (0,0);
\vertex (1,0);
\vertex (2,0);
\vertex (3,0);
\vertex (4,0);
\vertex (5,0);
\vertex (6,0);
\vertex (7,0);
\vertex (3,1);
\end{tikzpicture}%
}}%
\put(90,9){\makebox[0pt][l]{
\begin{tikzpicture}[line width=\mywidth, scale=\myscale]
\draw 
(180:2.307) -- 
(180:1.307) -- 
(135:1.307) -- 
(90:1.307) -- 
(45:1.307) --
(0:1.307) --
(-45:1.307) --
(-90:1.307) -- 
(-135:1.307) --
(180:1.307);
\vertex (180:2.307); 
\vertex (180:1.307); 
\vertex (135:1.307); 
\vertex (90:1.307); 
\vertex (45:1.307);
\vertex (0:1.307);
\vertex (-45:1.307);
\vertex (-90:1.307); 
\vertex (-135:1.307); 
\end{tikzpicture}%
}}%
\put(0,28){\makebox[0pt][l]{rank 8}}
\put(15,27){\makebox[0pt][l]{
\begin{tikzpicture}[line width=\mywidth, scale=\myscale]
\draw (0,0) -- (5,0);
\draw (1,0) -- (1,1);
\draw (3,0) -- (3,1);
\vertex (0,0);
\vertex (1,0);
\vertex (2,0);
\vertex (3,0);
\vertex (4,0);
\vertex (5,0);
\vertex (1,1);
\vertex (3,1);
\end{tikzpicture}%
}}%
\put(45,27){\makebox[0pt][l]{
\begin{tikzpicture}[line width=\mywidth, scale=\myscale]
\draw (0,0) -- (5,0);
\draw (2,0) -- (2,2);
\vertex (0,0);
\vertex (1,0);
\vertex (2,0);
\vertex (3,0);
\vertex (4,0);
\vertex (5,0);
\vertex (2,1);
\vertex (2,2);
\end{tikzpicture}%
}}%
%
\put(75,26){\makebox[0pt][l]{
\begin{tikzpicture}[line width=\mywidth, scale=\myscale]
\draw 
(180:2.152) -- 
(180:1.152) -- 
(128:1.152) -- 
(77:1.152) -- 
(26:1.152) --
(-26:1.152) --
(-77:1.152) --
(-128:1.152) --
(180:1.152);
\vertex (180:2.152);
\vertex (180:1.152);
\vertex (128:1.152);
\vertex (77:1.152);
\vertex (26:1.152);
\vertex (-26:1.152);
\vertex (-77:1.152);
\vertex (-128:1.152);
\vertex (180:1.152);
\end{tikzpicture}%
}}%
\put(0,46){\makebox[0pt][l]{rank 7}}
\put(15,45){\makebox[0pt][l]{
\begin{tikzpicture}[line width=\mywidth, scale=\myscale]
\draw (0,0) -- (4,0);
\draw (1,0) -- (1,1);
\draw (2,0) -- (2,1);
\vertex (0,0);
\vertex (1,0);
\vertex (2,0);
\vertex (3,0);
\vertex (4,0);
\vertex (1,1);
\vertex (2,1);
\end{tikzpicture}%
}}%
\put(40,43){\makebox[0pt][l]{
\begin{tikzpicture}[line width=\mywidth, scale=\myscale]
\draw (-2,0)--(-1,0)--(-.5,.866)--(.5,.866)--(1,0)--(.5,-.866)--(-.5,-.866)--(-1,0);
\vertex (-2,0);
\vertex (-1,0);
\vertex (-.5,.866);
\vertex (.5,.866);
\vertex (1,0);
\vertex (.5,-.866);
\vertex (-.5,-.866);
\end{tikzpicture}%
}}%
\put(0,61.5){\makebox[0pt][l]{rank 6}}
\put(15,57){\makebox[0pt][l]{
\begin{tikzpicture}[line width=\mywidth, scale=\myscale]
\draw (0,0) -- (3,0);
\draw (2,-1) -- (2,1);
\vertex (0,0);
\vertex (1,0);
\vertex (2,0);
\vertex (3,0);
\vertex (2,1);
\vertex (2,-1);
\end{tikzpicture}%
}}%
\put(35,57.3){\makebox{
\begin{tikzpicture}[line width=\mywidth, scale=\myscale]
\draw (0,0) -- (-1,0);
\draw (0,0) -- (-.309,.951);
\draw (0,0) -- (-.309,-.951);
\draw (0,0) -- (.809,.588);
\draw (0,0) -- (.809,-.588);
\vertex (0,0);
\vertex (-1,0);
\vertex (-.309,.951);
\vertex (-.309,-.951);
\vertex (.809,.588);
\vertex (.809,-.588);
\end{tikzpicture}%
}}%
\put(49,58){\makebox[0pt][l]{
\begin{tikzpicture}[line width=\mywidth, scale=\myscale]
\draw (-1.851,0)--(-.851,0)--(-.263,.809)--(.688,.5)--(.688,-.5)--(-.263,-.809)--(-.851,0);
\vertex (-1.851,0);
\vertex (-.851,0);
\vertex (-.263,.809);
\vertex (-.263,-.809);
\vertex (.688,.5);
\vertex (.688,-.5);
\end{tikzpicture}%
}}%
\put(0,76){\makebox[0pt][l]{rank 5}}
\put(15,73){\makebox[0pt][l]{
\begin{tikzpicture}[line width=\mywidth, scale=\myscale]
\draw (0,0) -- (0,1.414) -- (1.414,1.414) -- (1.414,0) -- (0,0);
\draw (0,0) -- (1.414,1.414);
\vertex (0,0);
\vertex (0,1.414);
\vertex (1.414,0);
\vertex (1.414,1.414);
\vertex (.707,.707);
\end{tikzpicture}%
}}%
\put(28,73){\makebox[0pt][l]{
\begin{tikzpicture}[line width=\mywidth, scale=\myscale]
\draw (0,0) -- (.707,.707) -- (1.414,0) -- (.707,-.707) -- (0,0);
\draw (0,0) -- (-1,0);
\vertex (0,0);
\vertex (.707,.707);
\vertex (.707,-.707);
\vertex (1.414,0);
\vertex (-1,0);
\end{tikzpicture}%
}}%
\put(0,89){\makebox[0pt][l]{rank 4}}
\put(15,86){\makebox[0pt][l]{
\begin{tikzpicture}[line width=\mywidth, scale=\myscale]
\draw (0,1) -- (-.866,-.5) -- (.866,-.5) -- (0,1);
\draw (0,0) -- (0,1);
\draw (0,0) -- (-.866,-.5);
\draw (0,0) -- (.866,-.5);
\vertex (0,0);
\vertex (0,1);
\vertex (-.866,-.5);
\vertex (.866,-.5);
\end{tikzpicture}%
}}%
\put(30,87){\makebox{
\begin{tikzpicture}[line width=\mywidth, scale=\myscale]
\draw (0,0) -- (0,1) -- (1,1) -- (1,0) -- (0,0);
\draw (0,0) -- (1,1);
\vertex (0,0);
\vertex (0,1);
\vertex (1,0);
\vertex (1,1);
\end{tikzpicture}%
}}%
\put(40,87){\makebox[0pt][l]{
\begin{tikzpicture}[line width=\mywidth, scale=\myscale]
\draw (0,0) -- (.732,.5) -- (.732,-.5) -- (0,0);
\draw (0,0) -- (-1,0);
\vertex (0,0);
\vertex (.732,.5);
\vertex (.732,-.5);
\vertex (-1,0);
\end{tikzpicture}%
}}%
\end{picture}
\caption{The simply-laced hyperbolic Dynkin diagrams.  The {\it
  rank\/} means the number of nodes.}
\label{tab-diagrams}
\end{table}

For each generalized Cartan matrix $A$, Tits \cite{Tits} defined the
{\it Steinberg group}, a functor $\St_A$ from commutative rings to
groups that generalizes Steinberg's definition from the classical
finite dimensional case 
(the group $G'$ on p.~78 of \cite{Steinberg}).
Morita and Rehmann \cite{Morita-Rehmann}
give another definition, but it agrees with Tits' for the diagrams in
table~\ref{tab-diagrams} because these are $2$-spherical without
isolated nodes \cite[Remark ${\rm a}_4$, p.~550]{Tits}.  By taking a
quotient of $\St_A$, Tits defined another functor
$\G_{\!A}$ from commutative rings to groups, which we call the {\it Kac--Moody group}.  We call
$\St_A$ and $\G_{\!A}$ hyperbolic if $A$ is hyperbolic.
(Note: Tits actually defined a group functor $\widetilde{\G}_D$
for each root datum~$D$.  By $\G_{\!A}$ we mean $\widetilde{\G}_D$
where $D$ is the root datum which has generalized Cartan matrix~$A$
and is ``simply connected
in the strong sense'' \cite[p.~551]{Tits}.)

Tits showed that his model of a Kac--Moody group is the natural one,
at least for fields.  Namely: any group functor with some obviously
desirable properties admits a functorial homomorphism from $\G_{\!A}$,
which at every field is an isomorphism \cite[Thm.~1$'$, p.~553]{Tits}.
Tits does not call $\G_{\!A}$ a Kac--Moody group.  We call it this
just to have a name for the closest known approximation to whatever
the ultimate definition of ``the'' Kac--Moody functors will be.

Let $R$ be a commutative ring.  Tits' definition of $\St_A(R)$ is by a
presentation with a generator $X_\alpha(t)$ for each real root
$\alpha$ of the Kac--Moody algebra $\g_A$ and each $t\in R$.  Whenever
two real roots $\alpha,\beta$ form a {\it prenilpotent pair\/}
(defined in section~\ref{sec-Curtis-Tits-presentation}), Tits imposes
a relation $[X_\alpha(t),X_\beta(u)]=\cdots$ for each pair $t,u\in R$.
The right side is a product of other generators $X_\gamma(v)$,
generalizing the classical Chevalley relations; see
section~\ref{sec-Curtis-Tits-presentation} for the details in the
cases we need.  Unless $A$ has finite-dimensional type, there are
infinitely many Weyl-group orbits of prenilpotent pairs, yielding
infinitely many distinct kinds of relations.  Our main result is a
new, much simpler, presentation, given entirely in terms of the Dynkin
diagram:

\def\FOOalign#1#2{#1=\mathrlap{#2}\kern120pt}
\def\FOOtag#1{\rlap{\rm #1}\kern90pt}
\begin{table}
\begin{align*}
\left.
\begin{aligned}
\FOOalign{X_i(t)X_i(u)}{X_i(t+u)}\\
\FOOalign{[S_i^2,X_i(t)]}{1}\\
\FOOalign{S_i}{X_i(1) S_i X_i(1) S_i^{-1} X_i(1)}
\end{aligned}
\right\}\FOOtag{all $i$}
\\
\noalign{\medskip}
\left.
\begin{aligned}
\FOOalign{S_i S_j}{S_j S_i}\\
\FOOalign{[S_i,X_j(t)]}{1}\\
\FOOalign{[X_i(t),X_j(u)]}{1}
\end{aligned}
\right\}
\FOOtag{all unjoined $i\neq j$}
\\
\noalign{\medskip}
\left.
\begin{aligned}
\FOOalign{S_i S_j S_i}{S_j S_i S_j}\\
\FOOalign{S_i^2 S_j S_i^{-2}}{S_j^{-1}}\\
\FOOalign{X_i(t) S_j S_i}{S_j S_i X_j(t)}\\
\FOOalign{S_i^2 X_j(t) S_i^{-2}}{X_j(t)^{-1}}\\
\FOOalign{[X_i(t),S_i X_j(u) S_i^{-1}]}{1}\\
\FOOalign{[X_i(t),X_j(u)]}{S_i X_j(t u) S_i^{-1}}
\end{aligned}
\right\}
\FOOtag{all joined $i\neq j$}
\end{align*}
\caption{Defining relations for $\St_A(R)$ when $A$ is
  simply laced hyperbolic.  The generators are $X_i(t)$ and $S_i$ where
  $i$ varies over the nodes of the Dynkin diagram and $t$ and $u$ vary
  over $R$.  
See theorem~\ref{thm-explicit-presentation} for the 
additional relations needed to define $\G_{\!A}(R)$, and
corollary~\ref{cor-presentation-over-Z} for 
simplifications in the special case $R=\Z$.}
\label{tab-explicit-presentation}
\end{table}

\begin{theorem}[Presentation of Steinberg and Kac--Moody groups]
\label{thm-explicit-presentation}
Suppose $R$ is a commutative ring and $A$ is a simply laced hyperbolic
Dynkin diagram, with $I$ being its set of nodes.  Then the Steinberg
group
$\St_A(R)$ has
a presentation with generators $S_i$ and $X_i(t)$, with $i$ varying
over $I$ and $t$ over $R$, and relations listed in table~\ref{tab-explicit-presentation}.

The Kac--Moody group $\G_{\!A}(R)$ is the quotient of $\St_A(R)$ by the extra relations
$\htilde_i(a)\htilde_i(b)=\htilde_i(ab)$, for any single  $i\in I$ and all
units $a,b$ of $R$, where
$\htilde_i(a){}:=\stilde_i(a)\stilde_i(-1)$ and
$\stilde_i(a){}:=X_i(a) S_i X_i(1/a) S_i^{-1} X_i(a)$.
\end{theorem}

Our generating set coincides with the one in \cite{Carbone-Garland},
and the presentation works just as well for the simply laced spherical
or affine Dynkin diagrams without $A_1$ components; see
\cite{Allcock-affine}.  When $R=\Z$ the presentation simplifies
considerably.  We give it explicitly because this entire paper grew
from trying to understand the Kac--Moody group~$\G_{\!E_{10}}(\Z)$:

\begin{corollary}[Presentation over $\Z$]
\label{cor-presentation-over-Z}
If $A$ is simply laced hyperbolic, then the Steinberg group
$\St_A(\Z)$ has a presentation with generators $S_i$ and $X_i$, where
$i$ varies over the simple roots, and the relations listed in~Table~\ref{tab-explicit-presentation-Z}.  The Kac--Moody group
$\G_{\!A}(\Z)$ is the quotient of $\St_A(\Z)$ by the relation
$\htilde_i(-1)^2=1$, for any  single $i\in I$. \qed
\end{corollary}

\begin{table}
\begin{align*}
\left.
\begin{aligned}
\FOOalign{[S_i^2,X_i]}{1}\\
\FOOalign{S_i}{X_i S_i X_i S_i^{-1} X_i}
\end{aligned}
\right\}\FOOtag{all $i$}
\\
\noalign{\medskip}
\left.
\begin{aligned}
\FOOalign{S_i S_j}{S_j S_i}\\
\FOOalign{[S_i,X_j]}{1}\\
\FOOalign{[X_i,X_j]}{1}
\end{aligned}
\right\}
\FOOtag{all unjoined $i\neq j$}
\\
\noalign{\medskip}
\left.
\begin{aligned}
\FOOalign{S_i S_j S_i}{S_j S_i S_j}\\
\FOOalign{S_i^2 S_j S_i^{-2}}{S_j^{-1}}\\
\FOOalign{X_i S_j S_i}{S_j S_i X_j}\\
\FOOalign{S_i^2 X_j S_i^{-2}}{X_j^{-1}}\\
\FOOalign{[X_i,S_i X_j S_i^{-1}]}{1}\\
\FOOalign{[X_i,X_j]}{S_i X_j S_i^{-1}}
\end{aligned}
\right\}
\FOOtag{all joined $i\neq j$}
\end{align*}
\caption{Defining relations for $\St_A(\Z)$ when $A$ is
  simply laced hyperbolic; see corollary~\ref{cor-presentation-over-Z}.}
\label{tab-explicit-presentation-Z}
\end{table}

\begin{remark}
$X_i(u)$ in theorem~\ref{thm-explicit-presentation} corresponds to
  $X_i^u$ here; in particular $X_i=X_i(1)$.  Also, one
  can show that $\htilde_i(-1)=S_i^{-2}$ in $\St_A$.
  So one could rewrite the relation $\htilde_i(-1)^2=1$ as $S_i^4=1$.
\end{remark}

The next result follows from the evident fact that each relation
in table~\ref{tab-explicit-presentation} involves at most $2$ subscripts.  If $R$ is a field then the
$\G_{\!A}$ case is a special case of a result of Abramenko--M\"uhlherr
\cite{Abramenko-Muhlherr}\cite{Muhlherr}; see also \cite{Caprace}.

\begin{corollary}[Curtis--Tits property of the presentation]
\label{cor-Steinberg-as-direct-limit}
Let $R$ be a commutative ring and $A$ a simply laced hyperbolic Dynkin
diagram.  Consider the Steinberg groups $\St_B(R)$ and the obvious maps between
them, as $B$ varies over the singletons and pairs of nodes of $A$.
The direct limit of this family of groups equals the Steinberg group $\St_A(R)$.
The same result holds with $\G_{\!A}$ in place of $\St_A$ throughout.
\qed
\end{corollary}

As one might expect, this result allows one to deduce
finite-pre\-sen\-ta\-tion results about $\St_A(R)$ from similar
results about the groups $\St_B(R)$.  The following theorem follows
immediately from theorem~\ref{thm-Pst-to-St-an-isomorphism} in the current paper (a restatement of
theorem~\ref{thm-explicit-presentation}) and Thm.~1.4 of~\cite{Allcock-amalgams}.  Of course,
any finite presentation result will require some hypothesis on~$R$.
But conceptually one might think of the presentation in table~\ref{tab-explicit-presentation}
as ``finite over~$R$'' for any commutative ring~$R$.  By this we mean
that the generators and relations have finitely many forms,
with some of the forms being parameterized by elements
of~$R$ (or pairs of elements).


\begin{theorem}[Finite presentation]
\label{thm-finite-presentation-of-Steinberg}
In the setting of corollary~\ref{cor-Steinberg-as-direct-limit},
$\St_A(R)$ is fi\-ni\-tely presented as a group if either
\begin{enumerate}
\item
\label{item-module-finite-over-suitable-subring}
\leavevmode
$R$ is finitely generated as a module over some subring
generated by finitely many units, or
\item
\label{item-rk-3-and-finitely-generated-ring}
\leavevmode
\hbox{$\rk A>4$} and $R$ is finitely generated as a ring.
\end{enumerate}
In either case, if the unit group of $R$ is finitely generated as an
abelian group, then $\G_{\!A}(R)$ is also finitely presented as a group.
\qed
\end{theorem}

Many mathematicians have worked on the question of whether
$S$-arithmetic groups in algebraic groups over adele rings are
finitely presented.  This was finally resolved in all cases by Behr
\cite{Behr-number-field-case}, \cite{Behr-function-field-case}.  Since
Kac--Moody groups are infinite-dimensional analogues of algebraic
groups, it is natural to ask whether their ``$S$-arithmetic groups''
are finitely presented.  The following result answers this, at least
insofar as $\G_{\!A}$ is an analogue of an algebraic group.  It is an
immediate application of
theorem~\ref{thm-finite-presentation-of-Steinberg}.

\begin{corollary}[Finite presentation in arithmetic contexts]
\label{cor-Kac-Moody-finite-presentation}
Suppose $K$ is a global field, meaning a finite extension of $\Q$ or
$\F_q(t)$.  Suppose $S$ is a nonempty finite set of places of $K$,
including all infinite places in the number field case.  Let $R$ be
the ring of $S$-integers in $K$.

Suppose $A$ is a simply laced hyperbolic Dynkin diagram.
Then 
$\G_{\!A}(R)$ 
and
$\St_A(R)$ 
are finitely presented, except perhaps in the case that
$\rk A=4$, $K$ is a function field and $|S|=1$.
\qed
\end{corollary}

Higher-dimensional finiteness properties are also very interesting.
Their analysis for most $S$-arithmetic subgroups of algebraic groups
has recently been completed by Bux, K\"ohl and Witzel
\cite{Bux-et-al}.  One should be able to combine their results with
corollary~\ref{cor-Steinberg-as-direct-limit} to obtain
higher-dimensional finiteness properties in the setting of
corollary~\ref{cor-Kac-Moody-finite-presentation}.  

\medskip
A major motivation for this work came from the conjectural appearance
of integral forms of hyperbolic Kac--Moody groups as symmetries of
supergravity and superstring theories \cite{Allcock-Carbone}.  $E_{10}$ is the ``overextended'' version of
$E_8$, and the corresponding overextended versions of $E_6$ and $E_7$
also appear in table~\ref{tab-diagrams}.  Hull and Townsend
conjectured that $\G_{\!E_{10}}(\mathbb{Z})$ is  the discrete
``U-duality'' group of Type II superstring theories \cite{Hull-Townsend}.  And by analogy with $SL_2(\mathbb{Z})$, Damour and Nicolai
conjectured that $\G_{\!E_{10}}(\mathbb{Z})$ is the ``modular group'' for
certain automorphic forms that are expected to arise in the context of
$11$-dimensional supergravity \cite{Damour-Nicolai}.   
The role of $E_{10}$ in the physics conjectures is somewhat mysterious and not well understood. We began this work by
pondering how to give a ``workable'' definition of  $\G_{\!E_{10}}(\Z)$.   We hope that our explicit
finite presentation will provide insight into these conjectures.

Our other major motivation was to bring Kac--Moody groups into the
world of geometric and combinatorial group theory, leading to
many new open questions such as those raised in \cite{Allcock-affine}.

\medskip
The methods of this paper use hyperbolic geometry in an essential way.
In particular, the proof of theorem~\ref{thm-Pst-to-St-an-isomorphism}
relies on distance estimates in hyperbolic space.  Therefore our
proofs do not extend to general Kac--Moody groups.  The simply-laced
hypothesis could probably be removed at the cost of additional
hypotheses on~$R$, since double and triple bonds are known to cause
complications over~$\F_2$.  See \cite{Abramenko-Muhlherr} for this, in
particular for the suggestion that a Kac--Moody group over $\F_2$
might fail to be finitely presented when all the nodes of its Dynkin
diagram are joined to each other by double bonds.  It is not clear yet
how well the results of this paper will extend to the general case.
But the first author has been able treat some Kac--Moody groups beyond
the hyperbolic cases of this paper.  These results will appear
separately.

\medskip
The first author is very grateful to the Japan Society for the
Promotion of Science. Both authors are very grateful to Kyoto
University for its support and hospitality.

\section{The Curtis--Tits presentation for Steinberg groups}
\label{sec-Curtis-Tits-presentation}

\noindent
We fix a simply laced hyperbolic Dynkin diagram~$A$
and a commutative ring~$R$.  We will briefly review Tits' definition of
$\St_A(R)$, recall the ``pre-Steinberg group'' $\PSt_A$ from
\cite{Allcock-amalgams}, prove that $\PSt_A\to\St_A$ is an isomorphism (theorem~\ref{thm-Pst-to-St-an-isomorphism}), and deduce
theorem~\ref{thm-explicit-presentation}.

Regarding $A$ as a generalized Cartan matrix, it is symmetric.  So we
may regard it as the inner product matrix of the simple roots
and then extend linearly to the root lattice~$\L$.  We indicate this
inner product by ``$\cdot$''.  By the {\it norm}, $\alpha^2$, of an element $\alpha\in\L$,
  we mean $\alpha\cdot\alpha$.  The simple roots $\alpha_i$ have
norm~$2$,  and reflections $w_{\alpha_i}$ in the $\alpha_i$ are isometries of $\L$.  The Weyl group $W$
is the group generated by the $w_{\alpha_i}$.  The $W\!$-images of the simple roots are
called real roots, and we write $\Phi$ for the set of
them. Since the inner product is $W\!$-invariant, all real roots have norm~$2$.

The group $\St_{A}(R)$ is defined as a certain quotient of the free
product $\freeproduct_{\alpha\in\Phi}\U_\alpha$, where $\U_\alpha$ is
a copy of the additive group of $R$.  A standard difficulty in Lie
theory is that it is impossible to distinguish a single ``best''
isomorphism $\U_\alpha\iso R$.  Instead, there is a natural pair of
parameterization $R\to\U_{\alpha}$, differing by inversion.  (Tits
refers to the ``double bases'' of the root spaces
\cite[\S3.3]{Tits}.)  For each $\alpha\in\Phi$ we fix one of these
isomorphisms and call it $X_\alpha$, so the $X_\alpha(t)$ with $t\in
R$ are the elements of $\U_\alpha$, with the obvious group operation.

The sign in lemma~\ref{lem-prenilpotent-pairs} below depends on this
choice, but only in a way that won't affect any of our arguments.  We
remark also that our presentation in
table~\ref{tab-explicit-presentation} does not involve a choice of
$X_\alpha$ for every $\alpha\in\Phi$.  We made such a choice only for
the simple roots $\alpha_i$.  This choice for the simple roots does not
distinguish any ``natural'' choices for the other roots.  For example,
if $i$ and $j$ are  joined, then $t\mapsto S_i X_j(t) S_i^{-1}$ and
$t\mapsto S_j X_i(t) S_j^{-1}$ are the two possibilities for
$X_{\alpha_i+\alpha_j}$, which by symmetry are equally preferable.
Happily, we do not need our choices to be natural in any way: one may choose
the $X_{\alpha\in\Phi}$ arbitrarily.

Describing the relations requires a preliminary definition.  Let
$(\alpha,\beta)$ be a pair of real roots. Then $(\alpha,\beta)$ is
called a {\it prenilpotent pair}  if  there exist $w,\ w'\in W$ such that
$$w\alpha, \ w\beta\in\Phi_+{\text{ and }}w'\alpha, \ w'\beta\in\Phi_-.$$
One can show that a pair of real roots $\{\alpha,\beta\}$ is prenilpotent if and only if $\alpha\neq -\beta$ and 
$$(\mathbb{N} \alpha + \mathbb{N} \beta )\cap \Phi$$ is a finite set (see for example \cite{CR}).

The relations in $\St_A(R)$ are the
following: for each prenilpotent pair $\alpha,\beta$ and each pair
$t,u\in R$, there is a relation 
\begin{equation}
\label{eq-Chevalley-relation}
[X_\alpha(t),X_\beta(u)]=\prod_{\gamma\in\Phi\cap(\N\alpha+\N\beta)-\{\alpha,\beta\}}X_\gamma(v),
\end{equation}
where the $v$'s on the right side depend on 
$\alpha,\beta,t,u$, the ordering on the $\gamma$'s, and the chosen
isomorphisms from $R$ to  
$\U_\alpha$, $\U_\beta$ and the $\U_\gamma$'s.
(A consequence of prenilpotency is that the product has only finitely
many factors.)
The following lemma describes the prenilpotent pairs in our situation,
and makes these relations explicit.  $\St_A(R)$ is the quotient of
$\freeproduct_{\alpha\in\Phi}\U_\alpha$ by all of these relations.

\begin{lemma}
\label{lem-prenilpotent-pairs}
Suppose $A$ is simply laced and hyperbolic.  Then
distinct real roots $\alpha,\beta\in\Phi$ form a prenilpotent pair just if
$\alpha\cdot\beta\geq-1$.  The corresponding relations in $\St_A(R)$
are
\begin{equation*}
[X_\alpha(t),X_\beta(u)]
=
\begin{cases}
X_{\alpha+\beta}(\pm t u)
&
\hbox{\rm if $\alpha\cdot\beta=-1$}
\\
1
&
\hbox{\rm otherwise}
\end{cases}
\end{equation*}
for all $t,u\in R$.  
\end{lemma}

We remark that if $\alpha\cdot\beta=-1$ then $\alpha+\beta$ is also a
real root, so the first case makes sense.  The sign in
$X_{\alpha+\beta}(\pm t u)$ depends on the choices of isomorphisms 
$X_\alpha$, $X_\beta$, $X_{\alpha+\beta}$ from $R$ to 
$\U_\alpha$, $\U_\beta$, $\U_{\alpha+\beta}$.  But we will not
use the relation itself, merely the fact that
$\U_{\alpha+\beta}=[\U_\alpha,\U_\beta]$.

We will use the following special features of the hyperbolic case.
First the signature of $\L$ is $(\rk A-1,1)$, so the vectors in
$\L\tensor\R$ of norm${}<0$ fall into two components.  
The fundamental chamber
$$
C:=\set{x\in\L\tensor\R}{\hbox{$x\cdot\alpha_i\leq0$ for all $i$}}
$$ meets only one of these components, which we call the future
cone~$F$.  The projectivization of $F$ is a copy of real hyperbolic
space of dimension $n:=\rk A-1$, for which we write $H^n$.  Second,
$C$ lies in the closure $\Fbar$, and its projectivization $PC$ is a
hyperbolic simplex together with its ideal vertices.  The reason for
this is that the Coxeter diagram underlying~$A$ is that of a
finite-covolume hyperbolic reflection group; see
\cite[\S\S6.8--6.9]{Humphreys}.  The Weyl group $W$ is defined as the
group generated by the reflections in the simple roots, and the Tits
cone is defined as the union of the $W$-images of~$C$.  We can now
state the third special property: the interior of the Tits cone is
exactly the future cone.  One direction is obvious: since
$C\sset\Fbar$ and $W$ preserves the open set $F$, the interior of the
Tits cone lies in $F$.  For the other direction, one must show that
the $W$-translates of $C\cap F$
cover~$F$.  This is part of Poincar\'e's polyhedron theorem.  A 
very clean treatment in the case of reflection groups appears
in \cite[Thm.\ 60]{de-la-Harpe}.

\begin{proof}[Proof of lemma~\ref{lem-prenilpotent-pairs}]
First we use the coincidence  of $F$ with the Tits cone's interior to
rephrase prenilpotency as follows. {\it Claim: real roots
  $\alpha,\beta$ form a prenilpotent pair if and only if some vector
  of $F$ has positive inner product with both of them, and some other
  vector of $F$ has negative inner product with both of them.}  
This follows from the following observation.
Fix a
vector~$v$ in the interior of $C$ and recall that $\Phi^+$ consists of
the real roots having negative product with it, and similarly for
$\Phi^-$.  Then for any $w\in W$, $w$ sends $\alpha,\beta$ into $\Phi^+$
(resp.\ $\Phi^-$) if
and only if $\alpha,\beta$ have negative (resp.\ positive) inner product with
$w^{-1}(v)\in F$.

Now suppose $\alpha,\beta\in\Phi$.  Since $\L\tensor\R$ has signature
$(n,1)$, $\alpha^\perp$ and $\beta^\perp$ meet in $F$ just if the
inner product matrix of $\alpha$ and $\beta$ is positive definite,
i.e., just if $\alpha\cdot \beta\in\{0,\pm1\}$.  (Here ${}\perp$
indicates the orthogonal complement in $\L\tensor\R$.)  In this case
$\alpha^\perp$ and $\beta^\perp$ are transverse at a point of
$\alpha^\perp\cap\beta^\perp\cap F$.  Obviously we may choose a nearby
element of $F$ on the positive side of both, and another element of $F$
on the negative side of both.  So in this case the pair is prenilpotent.

If $\alpha\cdot\beta\leq-2$ then their positive half-spaces in $F$ are
disjoint.  Therefore it is impossible to choose a point of $F$ that is
on the positive side of both $\alpha^\perp$ and $\beta^\perp$.  So the
pair is not prenilpotent.  

If $\alpha\cdot
\beta\geq2$ then one positive half-space in $F$ lies inside the other,
so obviously there is a point in the intersection, and similarly in
the intersection of the negative half-spaces.  So the pair is
prenilpotent.  This finishes the proof of the first claim.  

If $\alpha\cdot\beta=-1$, then $\alpha,\beta$ are simple roots for an
$A_2$ root system and the displayed relation is the corresponding
Chevalley relation.  If $\alpha\cdot\beta\geq0$ then the only roots
$\gamma$ in $\N\alpha+\N\beta$ are $\alpha$ and $\beta$ (indeed these
are the only vectors of norm~$\leq2$).  Since the product on the right
side of the Chevalley relation \eqref{eq-Chevalley-relation} is empty,
the
relation is $[\U_\alpha,\U_\beta]=1$.
\end{proof}

We mentioned above that $PC\sset\Hnbar$ is a hyperbolic simplex.   
Its facets have a curious
geometric property that turns out to be the key to our
proof of theorem~\ref{thm-explicit-presentation}.  We have not seen anything like it in Kac--Moody
theory before.  Recall that the hyperbolic distance between two points
of $H^n$, represented by vectors $x,y\in F$, is
$\cosh^{-1}\sqrt{(x\cdot y)^2/x^2y^2}$.  

\begin{lemma}
\label{lem-curious-property-of-facets}
Suppose $\phi$ is any facet of the projectivized Weyl chamber $PC$ and $p$
is any point of $\phi$.  Then there is a facet $\phi'$ of $PC$ that
makes angle $\pi/3$ with $\phi$, such that the hyperbolic distance 
$d(p,\phi\cap\phi')$ is at most $\cosh^{-1}\sqrt{4/3}\approx.549$.
\end{lemma}


\begin{proof}
We applied the following argument to each of the~$18$ possibilities
for~$A$.  We write $I$ for the set of $A$'s nodes, and for $i\in I$ we
write $\phi_i$ for the corresponding facet of~$PC$.  We also fix
elements $\omega_i\in\L\tensor\Q$ with $\omega_i\cdot\alpha_j=-1$ or~$0$ 
according to whether $i,j\in I$ are equal or not.  In Lie terminology
these are the fundamental weights.  Geometrically, $\omega_i$ represents
the vertex of $PC$ opposite $\phi_i$.  

We applied the following argument to each of the $\rk A$ many
possibilities for $\phi:=\phi_i$.  We  write
$J$ for the set of $j\in I$ that are joined to $i$.  We define $q$ as
the point of~$H^n$ represented by the sum of the $\omega_{j\in J}$.  It
lies in the interior of the face of $\phi$ that is opposite (in
$\phi$) to the face $\phi\cap\bigl(\cap_{j\in J}\phi_j\bigr)$ of~$\phi$.  For
each $j\in J$ we write $K_j$ for the convex hull of $q$ and
$\phi\cap\phi_j$.  By the property of $q$ just mentioned, $PC$ is the
union of the $K_j$.  For every $j\in J$ we found by inner product
computations that $\cosh^2\bigl(d(q,\phi\cap\phi_j)\bigr)\leq4/3$.
In surprisingly many cases we found equality.  We
used the PARI/GP package \cite{PARI} for the 
calculations.

Now, given $p\in\phi$, it lies in $K_j$ for some $j\in J$, and we set
$\phi'=\phi_j$.  By  $K_j$'s definition, $q$ is its point furthest
from $\phi\cap\phi'$.   So 
$$
d(p,\phi\cap\phi')
\leq d(q,\phi\cap\phi')\leq\cosh^{-1}\sqrt{4/3}.
$$
\end{proof}

At the beginning of this section we mentioned the {\it pre-Steinberg
  group} $\PSt_A$.  It is a group functor defined in
\cite{Allcock-amalgams}, by the same definition as the Steinberg group,
except that we only impose the Chevalley relations for prenilpotent
pairs $\alpha,\beta$ that are {\it classically prenilpotent\/}.  This
means that 
\hbox{$(\Q\alpha+\Q\beta)\cap\Phi$} is finite and $\alpha+\beta\neq0$.
Arguing as in the proof of lemma~\ref{lem-prenilpotent-pairs} shows
that this is equivalent to $\alpha=\beta$ or 
$\alpha\cdot\beta\in\{0,\pm1\}$.  
After stating the following theorem about $\PSt_A$, we will show how
theorem~\ref{thm-explicit-presentation} reduces to it.  Then we will prove it.

\begin{theorem}
\label{thm-Pst-to-St-an-isomorphism}
The natural map $\PSt_A(R)\to\St_A(R)$ is an isomorphism.
\end{theorem}

\begin{proof}[Proof of theorem~\ref{thm-explicit-presentation}, given theorem~\ref{thm-Pst-to-St-an-isomorphism}]
Theorem~1.2 of \cite{Allcock-amalgams} gives an
explicit presentation for $\PSt_A(R)$, namely the one in
table~\ref{tab-explicit-presentation}.  So theorem~\ref{thm-Pst-to-St-an-isomorphism}  is
identical to the first part of
theorem~\ref{thm-explicit-presentation}.

For the second part, we recall that Tits defined $\G_{\!A}(R)$ as the
quotient of $\St_A(R)$ by the relations
$\htilde_i(a)\htilde_i(b)=\htilde_i(ab)$ for all $i\in I$ and all
units $a,b$ of~$R$.  In fact imposing these relations for a single $i$
automatically gives the others too.  This follows from the fact that all roots are equivalent under the Weyl group, because $A$ is simply laced.
\end{proof}

\begin{proof}[Proof of theorem~\ref{thm-Pst-to-St-an-isomorphism}]
We
must show that the Chevalley relations for classically prenilpotent
pairs imply all the other Chevalley relations.
We will abbreviate $\PSt_A(R)$ and $\St_A(R)$ to $\PSt$ and $\St$.

By lemma~\ref{lem-prenilpotent-pairs}, any prenilpotent pair has inner
product${}\geq-1$.  Therefore it suffices to prove the following by
induction on $k\geq-1$: for every prenilpotent pair
$\alpha,\beta\in\Phi$ with $\alpha\cdot\beta=k$, the Chevalley
relations of $\alpha$ and $\beta$ in $\St$ already hold in $\PSt$.  
If $\alpha=\beta$ then the Chevalley relations say merely that
$\U_\alpha$ is commutative, which follows from the
multiplication rules in $\U_\alpha$. So we will suppose
$\alpha\neq\beta$.
If
$k\in\{0,\pm1\}$ then the pair is classically
prenilpotent, so this holds by definition of~$\PSt$.

So suppose $k:=\alpha\cdot\beta\geq2$.
By lemma~\ref{lem-prenilpotent-pairs}, 
the Chevalley relation we must establish
is
$[\U_\alpha,\U_\beta]=1$. 
We will exhibit roots $\alpha',\alpha''$ with 
\begin{equation}
\label{eq-properties-of-alpha-prime-and-alpha-double-prime}
\alpha'+\alpha''=\alpha,
\qquad
\alpha'\cdot\beta>0,
\quad\hbox{and}\quad
\alpha''\cdot\beta>0.   
\end{equation}
It follows that both these
inner products are less than~$k$.  By induction we get
$[\U_{\alpha'},\U_\beta]=[\U_{\alpha''},\U_{\beta}]=1$.  Since
$\alpha=\alpha'+\alpha''$, we have
$\U_\alpha=[\U_{\alpha'},\U_{\alpha''}]$.  Since $\U_\beta$ commutes
with $\U_{\alpha'}$ and $\U_{\alpha''}$, it also commutes with
$\U_{\alpha}$, as desired.  

It remains to construct $\alpha'$ and $\alpha''$.  We distinguish two
cases: $k=2$ and $k>2$.
First suppose $k=2$.  
What is special about this case is that the span of $\alpha$ and
$\beta$ is degenerate: $\nu:=\alpha-\beta$ is orthogonal to both and
has norm~$0$.  It
will suffice to exhibit $\alpha'$ with
$\alpha'\cdot\alpha=\alpha'\cdot\beta=1$, for then we can take
$\alpha'':=\alpha-\alpha'$ and apply the previous paragraph.

Since $\Fbar\cup(-\Fbar)$ is the set of vectors of
norm${}\leq0$, we have
$\nu\in\pm\Fbar$.  By
exchanging $\alpha,\beta$ we may suppose without loss that $\nu\in\Fbar$.
We fix a vector $x\in\Lambda$ in the interior of $C$ and consider its
inner products with the images of $\nu$ under the Weyl group~$W$.
These inner products are integral because $x\in\Lambda$, and nonpositive
because $x\cdot F\sset(-\infty,0)$.  Therefore they achieve their
maximum, which is to say: after replacing $\alpha,\beta$ by their
images under an element of~$W$ we may suppose without loss that
$\nu\cdot x=\max_{w\in W}w(\nu)\cdot x$.  This maximality forces
$\nu\cdot\alpha_i\leq0$ for all~$i$, for if $\nu\cdot\alpha_i>0$ then
reflection in $\alpha_i$ increases $\nu$'s inner product with $x$.  So
$\nu\in C$.   That is, it
represents an ideal vertex of $PC$.  

The simple roots orthogonal to $\nu$ correspond to the nodes of an
irreducible affine subdiagram $A^\nu$ of~$A$, and the $W$-stabilizer
of $\nu$ is the Weyl group of $A^\nu$.  By replacing $\alpha,\beta$ by
their images under an element of this stabilizer, we may suppose
without loss that $\alpha$ is a simple root corresponding to a node
of $A^\nu$.  Since every node of $A^\nu$ is joined to some other node
of $A^\nu$, $\alpha$ lies in some $A_2$ root system in $A^\nu$.  Inside this $A_2$ is a root $\alpha'$ having inner
product~$1$ with~$\alpha$.  Since $\alpha'$ is orthogonal to~$\nu$, it also
has inner product~$1$ with~$\beta$.  This finishes the case $k=2$.

Now for the inductive step: fixing $k\geq3$, we seek
roots $\alpha',\alpha''$ with the properties \eqref{eq-properties-of-alpha-prime-and-alpha-double-prime}.  
Regarding $\alpha^\perp$ and $\beta^\perp$
as hyperplanes in $H^n$, we define $p,q\in H^n$ as follows.
First, $p$ is the point of $\alpha^\perp$
closest to $\beta^\perp$.  Second, $q$ is a point of $\alpha^\perp$,
which is orthogonal to some real root $\alpha'$ with
$\alpha'\cdot\alpha=1$, and closest possible to $p$  among all such points.  
Because $p$ lies on
some facet of some chamber, lemma~\ref{lem-curious-property-of-facets} shows that such an $\alpha'$
exists, so the definition of $q$ makes sense.  It also shows that $d(p,q)$ is at most $\cosh^{-1}\sqrt{4/3}$.

Together, $\alpha$ and $\alpha'$ span an $A_2$ root system, of which 
$\alpha'':=\alpha-\alpha'$ is another root.  
To prove $0<\beta\cdot\alpha'$ and
$0<\beta\cdot\alpha''$,  suppose otherwise, say
$\beta\cdot\alpha'=m\leq0$.  The idea is to work out $p$ and $q$
explicitly and find that $d(p,q)$ is larger than
$\cosh^{-1}\sqrt{4/3}$, which is a contradiction.  The inner product
matrix of $\alpha,\alpha',\beta$ is
$$
\begin{pmatrix}
2&1&k\\
1&2&m\\
k&m&2
\end{pmatrix}
$$ and (a vector representing) $p$ is the projection of $b$ to $a^\perp$, i.e., $p=\beta-\alpha(\beta\cdot
\alpha)/2$.  Now, $q$ is the projection of $p$ to $\alpha^\perp\cap\alpha'^\perp$,
or equivalently $\alpha^\perp\cap(\alpha''-\alpha')^\perp$.  The advantage of the
latter description is that $\alpha$ is orthogonal to $\alpha''-\alpha'$.  
Since $p$ is already orthogonal to $\alpha$, projecting it to
this codimension~$2$ subspace is the same as projecting it to
$(\alpha''-\alpha')^\perp$.  
For
calculations in our basis we note that $\alpha''-\alpha'$ has norm~$6$ and we
rewrite it as $\alpha-2\alpha'$.
So $q=p-(\alpha-2\alpha')\big(p\cdot(\alpha-2\alpha')\bigr)/6$.
Calculation reveals
$$
d(p,q)=\cosh^{-1}\sqrt{\frac{4}{3}\cdot\frac{3+km-k^2-m^2}{4-k^2}}
$$ By applying $\partial/\partial m$ to the radicand and using $m\leq0$
and $k\geq3$, one checks that the right side is decreasing as
a function of $m$.  Therefore $d(p,q)$ is at least what one would get by
plugging in $0$ for $m$:
$$
d(p,q)\geq\cosh^{-1}\sqrt{\frac{4}{3}\cdot\frac{3-k^2}{4-k^2}}
$$
By differentiating one shows that the right side
is decreasing in $k$.  So $d(p,q)$ is larger than the limit as
$k\to\infty$, which is $\cosh^{-1}\sqrt{4/3}$.  This is a
contradiction, proving the claim.
\end{proof}

\begin{remark}
The origin of the proof was picture-drawing in the hyperbolic plane
associated to the span of $\alpha,\alpha',\beta$.  We recommend
sketching the configuration of $\alpha^\perp$, $\alpha'^\perp$,
$\alpha''^\perp$ and $\beta^\perp$ when $m=0$, and contemplating how
it would change if $m$  were negative.  When $m=0$,  the quadrilateral spanned by $p,q$ and their projections to
$\beta^\perp$ converges to a $(2,3,\infty)$ triangle as $k\to\infty$.
So the constant $\cosh^{-1}\sqrt{4/3}$ is the length of the short edge
of the $(2,3,\infty)$ triangle in $H^2$.  It is curious that this
bound in lemma~\ref{lem-curious-property-of-facets}, which was optimal, is only barely
sufficient for the proof of theorem~\ref{thm-Pst-to-St-an-isomorphism}.
\end{remark}


\begin{thebibliography}{99}

\bibitem{Abramenko-Muhlherr} Abramenko, P. and M\"uhlherr, B.,
  {\it Pre\'sentations de certaines $BN$-paires jumel\'ees comme sommes
  amalgam\'ees}, {\it C.R.A.S. S\'erie~I} {\bf 325} (1997) 701--706.

\bibitem{Allcock-amalgams}
Allcock, D.,  {\it Steinberg groups as amalgams}, {\it arXiv:1307.2689}.

\bibitem{Allcock-affine}
Allcock, D., {\it Presentation of affine 
  Kac--Moody groups over rings}, 
{\it arXiv:\discretionary{}{}{}1409.\discretionary{}{}{}0176}.

\bibitem{Allcock-prenilpotent}
Allcock, D., {\it The prenilpotent pairs in the $E_{10}$ root system},
{\it arXiv:1409.0184}.

\bibitem{Allcock-Carbone} Allcock, D. and Carbone, L., {\it  A finite presentation for the hyperbolic Kac--Moody group $E_{10}$}, in preparation (2014).

\bibitem{Behr-number-field-case} Behr, H., \"Uber die endliche
  Definierbarkeit verallgemeinerter Einheitsgruppen II, {\it
    Invent. Math.} {\bf 4} (1967) 265--274.

\bibitem{Behr-function-field-case} Behr, H., Arithmetic groups over function fields I,
  {\it J. reine angew. Math.} {bf 495} (1998) 79--118.

\bibitem{Bux-et-al}
Bux, K.-U., K\"ohl, R. and Witzel, S.,
Higher finiteness properties of reductive arithmetic groups in
positive characteristic: the rank theorem, {\it 
Ann. Math. (2)} {\bf 177} (2013) 311--366.



\bibitem{Carbone-Garland} Carbone, L., {\it Infinite dimensional Chevalley
groups and Kac--Moody groups over $\Z$}, In preparation (2015)

\bibitem{Carbone-Murray} Carbone, L. and  Murray, S.,
{\it Prenilpotent pairs in hyperbolic Kac--Moody root systems}, In preparation (2015)

\bibitem{Caprace} Caprace, P.-E., {\it On 2-spherical Kac--Moody
  groups and their central extensions}, {\it Forum Math.} {\bf 19} (2007)
  763--781.

\bibitem{CR}
Caprace, P.-E. and R\'emy, B.
{\it Groups with a root datum}, Oxford Lecture Notes (2007)

\bibitem{Damour-Nicolai} Damour, T. and Nicolai, H. {\it Symmetries, singularities and the de-emergence of space}, International Journal of Modern Physics D. Vol 17, Nos 3 and 4 (2008) 525--531.

\bibitem{de-la-Harpe} de la Harpe, P., An invitation to Coxeter
  groups. {\it Group theory from a geometrical viewpoint (Trieste,
    1990)}, 193--253, World Sci. Publ., River Edge, NJ, 1991.

\bibitem{Hull-Townsend}  Hull, C. M. and  Townsend, P. K. {\it Unity of
  superstring dualities}, Nucl. Phys. B438:109--137, (1995)

\bibitem{Humphreys} Humphreys, J., {\it Reflection Groups and Coxeter
  Groups}, Cambridge Studies in Advanced Mathematics, 29. Cambridge
  University Press, Cambridge, 1990.


\bibitem{Morita-Rehmann} Morita, J. and Rehmann, U., {\it A Matusumoto--type
  theorem for Kac--Moody groups}, {\it T\^ohoku Math. J.} {\bf 42}
  (1990) 537--560.

\bibitem{Muhlherr} Muhlherr, B., {\it On the simple connectedness of a
  chamber system associated to a twin building}, unpublished (1999).

\bibitem{PARI}
    PARI/GP, version {\tt 2.5.3}, Bordeaux, 2013,
    {\it http://pari.math.u-bordeaux.fr/}

\bibitem{Splitthoff} 
Splitthoff, S.,
Finite presentability of Steinberg groups and related Chevalley
groups. in {\it Applications of algebraic K-theory to algebraic
  geometry and number theory, Part I, II (Boulder, Colo., 1983)}, pp. 635--687,
{\it Contemp. Math.} {\bf 55}, Amer. Math. Soc., Providence, RI, 1986. 

\bibitem{Steinberg} Steinberg, R., {\it Lectures on Chevalley Groups},
  lecture notes, Yale 1967.

\bibitem{Tits}
Tits, J., {\it Uniqueness and presentation of Kac--Moody groups over
fields}, {\it J. Algebra} {\bf 105} (1987) 542--573.

\end{thebibliography}
\end{document}